\documentclass[a4paper]{amsart}
\usepackage{amssymb}
\usepackage{amscd}
\usepackage{amsthm}
\usepackage{amsmath}
\usepackage{latexsym}
\usepackage[all]{xy}

\theoremstyle{plain}
\newtheorem{theorem}{Theorem}
\newtheorem{corollary}[theorem]{Corollary}
\newtheorem{lemma}[theorem]{Lemma}

\theoremstyle{definition}
\newtheorem{definition}[theorem]{Definition}
\newtheorem{example}[theorem]{Example}

\theoremstyle{remark}

\newtheorem{remark}[theorem]{Remark}

\numberwithin{theorem}{section}

\newcommand{\Spec}[1]{\mbox{\rm{Spec}}(#1)}
\newcommand{\mSpec}[1]{\mbox{\rm{mSpec}}(#1)}

\newcommand{\Add}{\mbox{\rm{Add\,}}}
\newcommand{\Prod}{\mbox{\rm{Prod\,}}}

\newcommand{\Cog}{\mbox{\rm{Cog\,}}}

\newcommand{\Ass}[1]{\mbox{\rm{Ass}} \,#1}

\newcommand{\im}{\mbox{\rm{Im\,}}}

\newcommand{\Hom}[3]{\mbox{\rm{Hom}}_{#1}(#2,#3)}
\newcommand{\Ext}[4]{\mbox{\rm{Ext}}^{#1}_{#2}(#3,#4)}
\newcommand{\Tor}[4]{\mbox{\rm{Tor}}_{#1}^{#2}(#3,#4)}

\newcommand{\rfmod}[1]{\mbox{\rm{mod}--}{#1}}
\newcommand{\rmod}[1]{\mbox{\rm{Mod}--}{#1}}

\newcommand{\Ker}[1]{\mbox{\rm{Ker}}(#1)}

\newcommand{\m}{\mathfrak{m}}
\newcommand{\p}{\mathfrak{p}}
\newcommand{\q}{\mathfrak{q}}

\begin{document}

\title{Colocalization and cotilting for commutative noetherian rings}

\author{Jan Trlifaj} 
\address{Charles University in Prague, Faculty of Mathematics and Physics, Department of Algebra \\
Sokolovsk\'{a} 83, 186 75 Prague 8, Czech Republic}
\email{trlifaj@karlin.mff.cuni.cz}

\author{Serap \c{S}ahinkaya}
\address{Gebze Institute of Technology, Department of Mathematics, 41400 Gebze, Turkey}
\email{ssahinkaya@gyte.edu.tr}

\date{\today}
\thanks{Research supported by GA\v CR 201/09/0816.}
\begin{abstract} For a commutative noetherian ring $R$, we investigate relations between tilting and cotilting modules in $\rmod R$ and $\rmod R_{\m}$ where $\m$ runs over the maximal spectrum of $R$. For each $n < \omega$, we construct a 1-1 correspondence between (equivalence classes of) $n$-cotilting $R$-modules $C$ and (equivalence classes of) compatible families $\mathcal F$ of $n$-cotilting $R_{\m}$-modules ($\m \in \mSpec R$). It is induced by the assignment $C \mapsto ( C^{\m} \mid \m \in \mSpec R )$ where $C^{\m} = \Hom R{R_{\m}}C$ is the colocalization of $C$ at $\m$, and its inverse $\mathcal F \mapsto \prod_{F \in \mathcal F} F$. We construct a similar correspondence for $n$-tilting modules using compatible families of localizations; however, there is no explicit formula for the inverse.   
\end{abstract}

\maketitle

\section*{Introduction}

Tilting and cotilting modules and classes over Dedekind domains are well understood for almost a decade, see \cite{BET}. A generalization of their classification to Pr\" ufer domains is due to Bazzoni \cite{SB1}, while tilting and cotilting classes over commutative noetherian rings have been characterized in \cite{APST}. Moreover, a construction of cotilting modules over commutative noetherian rings using injective precovers has recently been discovered in \cite{STH}.

Though tilting modules over commutative rings behave well with respect to localization \cite[\S 13.3]{GT}, this is not the case for cotilting modules. In fact, the localization of an injective cogenerator need not be injective or a cogenerator (see \cite{C} and \cite{STH}). The appropriate tool for studying cotilting modules over commutative rings is the colocalization: in \cite{SB1}, it was first used to transfer the classification problem for cotilting modules from Pr\" ufer to valuation domains. 

In the present paper, we extend this idea to commutative noetherian rings. Of course, there are essential differences between the Pr\" ufer and the noetherian cases: in the Pr\" ufer case, all cotiliting modules have injective dimension $\leq 1$, but they need not be of cofinite type (that is, equivalent to duals of tilting modules). In the noetherian case, there is no a priori bound on the injective dimension of cotilting modules, but all cotilting modules are of cofinite type \cite{APST}. This is the key point here; in fact, parts of our results on colocalization extend to cotilting modules of cofinite type over arbitrary commutative rings. 

In the noetherian setting, we construct a 1-1 correspondence between equivalence classes of cotilting modules $C$ and equivalence classes of compatible families of their colocalizations in all maximal ideals of $R$ (see Corollary \ref{coloccor}). We prove a similar results for tilting modules using families of localizations, but the reverse correspondence is much less direct in this case.

\section*{Preliminaries}

In what follows, $R$ is a commutative and associative ring with unit. The category of all ($R$-) modules is denoted by $\rmod R$, 
while $\rfmod R$ denotes the category of all \emph{strongly finitely presented} modules, that is the modules possessing a projective resolution consisting of finitely generated modules. For example, if $R$ is noetherian, then $\rfmod R$ is just the category of all finitely generated modules. The set of all prime (maximal) ideals of $R$ is denoted by $\Spec R$ ($\mSpec R$). For a module $M$, $\Add M$ denotes the class of all
direct summands of (possibly infinite) direct sums of copies of $M$. Similarly, $\Prod M$ denotes the class of all direct summands of  direct products of copies of $M$.

Recall that given a multiplicative set $S \subseteq R$ and a module $M \in \rmod R$, $S^{-1}R$ denotes the localization of $R$ at $S$, 
and $S^{-1}M \cong M\otimes S^{-1}R$ the localization of $M$ at $S$. 
In case $S = R \setminus \p$ for a prime ideal $\p$ of $R$, we will prefer the more standard notation of $R_{\p}$ and 
$M_{\p}$ for $S^{-1}R$ and $S^{-1}M$, respectively.

For a class of modules $\mathcal C$, we define the Ext-orthogonal classes $\mathcal C ^{\perp}=\{M \in \rmod R \mid \Ext iRCM =0 \mbox{ for all } C \in \mathcal C \mbox{ and } i \geq 1\}$ and 
$^{\perp} \mathcal C =\{M \in \rmod R \mid \Ext iRMC =0 \mbox{ for all } C \in \mathcal C \mbox{ and } i \geq 1\}$.

\medskip
Let $n < \omega$. A module $T$ is \emph{$n$-tilting} provided that
\begin{itemize}
\item[\rm{(T1)}] $T$ has projective dimension $\leq n$,
\item[\rm{(T2)}] $\Ext iRT{T^{(\kappa)}} = 0$ for each $i>0$ and all cardinals $\kappa$.
\item[\rm{(T3)}] There exists a long exact sequence
$0 \to R \to T_0 \to  T_1 \to \cdots \to T_n \to 0$
where $T_i \in \Add T$ for each $0\leq i \leq n$.
\end{itemize}

The class $\{ T \}^{\perp}$ is called the \emph{$n$-tilting class} induced by $T$.  If $T$ and $T'$ are tilting modules, then $T$ is said to be \emph{equivalent} to $T'$ provided
that $\{ T \}^{\perp}=\{ T' \}^{\perp}$.

Let $\mathcal{A},\mathcal{B} \in \rmod R$. The pair $(\mathcal{A},\mathcal{B})$ is called a (hereditary) \emph{cotorsion pair}, 
if $\mathcal{A}={}^{\perp}\mathcal{B}$ and $\mathcal{B}=\mathcal{A}^{\perp}$. Notice that for $\mathcal C \subseteq \rmod R$,
$(^{\perp}(\mathcal{C^{\perp}}),\mathcal{C^{\perp}})$ and $(^{\perp}\mathcal{C},(^{\perp}\mathcal{C})^{\perp})$
are cotorsion pairs, called the cotorsion pairs generated and cogenerated, respectively, by $\mathcal C$. 

For a class of modules $\mathcal X$, we let $\mathcal X^\intercal = \{M \in \rmod R \mid \Tor RiX{M} =0 \mbox{ for all } i \geq 1 \mbox{ and } X \in \mathcal X \}$. A subclass $\mathcal F$ of $\rfmod R$ is called \emph{resolving} provided that $\mathcal F$ is closed under extensions and direct summands, it contains $R$ as well as all kernels of epimorphisms in $\mathcal F$. For example, if $(\mathcal{A},\mathcal{B})$ is a cotosion pair, then the class $\mathcal A \cap \rfmod R$ is resolving.    

Given a class $\mathcal{F}\subseteq \rmod R$ and a module $M$, a well-ordered chain of submodules
$$0=M_0\subseteq M_1\subseteq \dots \subseteq M_{\alpha}\subseteq M_{\alpha+1}\subseteq \dots \subseteq M_\sigma=M$$
is called an $\mathcal{F}$-\emph{filtration} of $M$ if $M_{\beta}=\bigcup_{\alpha < \beta}M_{\alpha}$ for every limit ordinal
$\beta \leq \sigma$, and for each $\alpha < \sigma$, $M_{\alpha+1}/M_{\alpha}$ is isomorphic to an element of $\mathcal F$.
$M$ is $\mathcal{F}$-\emph{filtered} in case $M$ possesses at least one $\mathcal{F}$-filtration.

The following is known as the finite type theorem for tilting modules (see e.g.\ \cite[Theorem 13.46]{GT}):  

\begin{theorem}\label{ftype}  Let $R$ be a ring, $T$ a tilting module
and $(\mathcal{A},\mathcal{B})$ the cotorsion pair generated by $T$. Let $\mathcal F = \mathcal A \cap \rfmod R$. Then
$\mathcal F ^\perp = \mathcal B$, and $T$ is equivalent to a tilting module $T^\prime$ such that $T^\prime$  is $\mathcal F$-filtered.
\end{theorem}

The terminology of finite type here comes from the fact that in the theorem above, while $T$ is typically infinitely generated, the tilting class $\mathcal B$ equals $\mathcal F ^\perp$ where 
$\mathcal F$ is a resolving subcategory of $\rfmod R$ consisting of modules of bounded projective dimension (the bound being the projective dimension of $T$). 
 
If $T$ is an $n$-tilting module, then the localization $S^{-1} T$ is an $n$-tilting $S^{-1} R$-module
(see e.g.\ \cite[\S 13.3]{GT}). Moreover, filtrations go well with localization, that is, if $T$ is $\mathcal{C}$-filtered, 
then $S^{-1} T$ is $\mathcal{C}_S$-filtered where $\mathcal{C}_S=\{S^{-1} M \mid M \in \mathcal{C}\}$, and the induced tilting class in $\rmod {S^{-1}R}$ equals
$$\mathcal{T}_S=\{N \in \rmod {S^{-1}R} \mid N \cong S^{-1}M \mbox{ for some } M \in \mathcal{T} \}=\mathcal{C}_S^\perp$$
(see \cite[13.47 and 13.50]{GT}).
 
\medskip
We now recall the dual setting of cotilting modules. Let $n < \omega$. A module $C$ is \emph{$n$-cotilting} provided that
\begin{itemize}
\item[\rm{(C1)}] $C$ has injective dimension $\leq n$,
\item[\rm{(C2)}] $\Ext iR{C^{\kappa}}C = 0$ for each $i>0$ and all cardinals $\kappa$.
\item[\rm{(C3)}] There exists a long exact sequence
$0 \to C_n \to \cdots \to C_1 \to C_0 \to W \to 0$
where $C_i \in \Prod C$ for each $0\leq i\leq n$ and $W$ is an injective cogenerator for $\rmod R$.
\end{itemize}

By Bazzoni \cite{SB}, a module $C$ is $n$-cotilting, iff $^{\perp} C = \Cog{_n}(C)$ where $\Cog{_n}(C)$
denotes the class  of all modules $M$ for which there exists an exact sequence of the form
$$0 \to M \to C^{\alpha_0} \to C^{\alpha_2} \to \cdots \to C^{\alpha_{n-1}}$$ for some cardinals $\alpha_i$ ($i < n$).
The class ${}^\perp C$ is the \emph{$n$-cotilting class} induced by $C$. If $C$ and $C'$ are cotilting modules, then $C$ is said to be \emph{equivalent} to $C'$ provided
that ${}^{\perp} C ={}^{\perp}C'$.

A cotilting class $\mathcal{C}$ is of \emph{cofinite type} provided that there exist $n < \omega$ and a class $\mathcal{F}$ consisting of strongly finitely presented modules of projective dimension $\leq n$, such that $\mathcal{C}=\mathcal{F}^{\intercal}$. Let $C$ be a cotilting module. Then $C$ is of \emph{cofinite type} provided the class $^{\perp}C$ is of cofinite type. 

While all tilting modules are of finite type by Theorem \ref{ftype}, the dual result fails in general: cotilting modules that are not of cofinite type do exist over each non-strongly discrete valuation domain, \cite{SB1}. However, if $R$ is any commutative noetherian ring, then all cotilting modules are of cofinite type by \cite{APST} (see also Lemma \ref{apst} below). 

Let $\mu$ be an ordinal and $\mathcal{A}=(A_\alpha: \alpha \leq \mu)$ be a sequence of modules.
Let $(g_{\alpha\beta}:\alpha \leq \beta \leq \mu )$ be a sequence of epimorphisms with
$g_{\alpha\beta}\in \Hom R{A_\beta}{A\alpha}$, such that $\mathcal{E}=\{A_\alpha, g_{\alpha\beta} \mid \alpha \leq \beta \leq \mu\}$ is an inverse
system of modules. $\mathcal{E}$ is called \emph{continuous}, provided that $A_0=0$ and
$A_\alpha=\varprojlim_{ \beta \leq \alpha}A_\beta$ for all limit ordinals $\alpha \leq \mu$.

Let $\mathcal{C}$ be a class of modules. Assume $\mathcal{E}$ is a continuous inverse system as above.
Then $A = \mathcal{A}_\mu$ is called \emph{$\mathcal{C}$-cofiltered} (by $\mathcal{A}$), provided that $\Ker {g_{\alpha\alpha+1}}$ is isomorphic
to an element of $\mathcal{C}$ for all $\alpha \leq \mu$. The sequence $\mathcal{A}$ is a \emph{$\mathcal{C}$-cofiltration} of $A$.

Denote by $(-)^*$ the duality $\Hom R-W$ where $W$ is an injective cogenerator for $\rmod R$. If $T$ is an $n$-tilting module and $\mathcal F = {}^\perp (T^{\perp}) \cap \rfmod R$ then $\mathcal F$ is a resolving subcategory of $\rfmod R$ consisting of modules of projective dimension $\leq n$. W.l.o.g., $T$ is $\mathcal F$-filtered (see Theorem \ref{ftype}). The dual module $C=T^*$ is an $n$-cotilting module, and it is $\mathcal{F}^*$-cofiltered, the induced cotilting class in $\rmod R$ being $\mathcal{F} ^\intercal$ (see \cite[Theorem 15.2]{GT}). In fact, an $n$-cotilting module $D$ is of cofinite type, if and only if $D$ is equivalent to $T^*$ for an $n$-tilting module $T$. 

Finally, we note that unlike filtrations, the cofiltrations do not go well with localizations of modules because localizations do not commute with the inverse limits that are used in limit steps of constructions of the cofiltrations. Colocalizations (to be introduced in the next section) give a way of overcoming this problem.

\section{Colocalization of cotilting modules}

The notion of colocalization is due to Melkerson and Schenzel (cf.\ \cite[p.118]{X}), but similar constructions were used already in \cite{CE} (see Lemma \ref{cartanei} below).

\begin{definition} \label{defcoloc} For every module $M$ and every maximal ideal  $\m$ of $R$, we denote by
$M^{\m}$ the $R_{\m}$-module $\Hom R{R_{\m}}M$; we call it the \emph{colocalization} of $M$ at $\m$.

Similarly, for a class of modules $\mathcal C$, we define $\mathcal C ^{\m} = \{ C^{\m} \mid C \in \mathcal C \}$, so $\mathcal C ^{\m}$ is a class of $R_{\m}$-modules.
\end{definition}

\begin{example} \label{ex} 1. If $M$ is an $R_{\m}$-module, then 
$$M^{\m} = \Hom R{R_{\m}}M = \Hom {R_{\m}}{R_{\m}}M \cong M$$ as $R_{\m}$-modules.

2. If $M = N^*$ is a dual module, then $M^{\m} = \Hom R{R_{\m}}{N^*} \cong (N \otimes_R R_{\m})^*$, so the colocalization of a dual module is isomorphic to the dual of a localization.
In particular, $M^{\m}$ is pure-injective as an $R$-module for each pure-injective module $M$.
\end{example}

We will need the following classic facts:

\begin{lemma} \label{cartanei} Let $S$ be a multiplicative subset of $R$ and $i \geq 1$. 
\begin{itemize}
\item[\rm{(a)}] For every module $A$ and $S^{-1}R$-module $B$, we have $\Ext i{S^{-1}R}{S^{-1}A}B\cong \Ext iR{A}B$.
\item[\rm{(b)}] For every pure-injective module $B$ and every $S^{-1}R$-module $A$, we have $\Ext i{S^{-1}R}A {\Hom R{S^{-1}R}{B}} \cong \Ext iR{A}B$.
\end{itemize}
In particular, if $\m \in \mSpec R$, $M$ and $C$ are modules, and $C$ is pure-injective, then $\Ext iRM{C^{\m}} \cong \Ext iR{M_{\m}}C$.
\end{lemma}
\begin{proof} (a) Since $S^{-1}R$ is a flat module, $\Tor nR{S^{-1}R}A=0$ for every $n \geq 1$, and the claim follows by \cite[Proposition VI.4.1.3]{CE}.

(b) Since $S^{-1}R$ is flat and $B$ is pure-injective, $\Ext nR{S^{-1}R}B=0$ for every $n \geq 1$, and the claim follows by \cite[Proposition VI.4.1.4]{CE}.

Finally, applying the isomorphisms from (a) and (b) to the particular setting, we obtain $\Ext iRM{C^{\m}} \cong \Ext i{R_{\m}}{M_{\m}}{C^{\m}} \cong \Ext iR{M_{\m}}C$.
\end{proof}

In particular, Lemma \ref{cartanei}(b) applies to cotilting modules, because each cotilting module is pure-injective by \cite{S}.

The next Lemma shows that colocalization works fine in the setting of cotilting modules of cofinite type:  

\begin{lemma}\label{dual} Let $T$ be an $n$-tilting module, $\m \in \mSpec R$, and $W$ be an injective cogenerator for $\rmod R$. Let $(-)^* = \Hom R-W$ and $\mathcal F = {}^\perp(T^\perp) \cap \rfmod R$. 

Then dual module $C = T^*$ is an $n$-cotilting $R$-module (of cofinite type), and its colocalization $C^{\m}$ is an $n$-cotilting $R_{\m}$-module isomorphic to the dual of the $n$-tilting $R_{\m}$-module $T_{\m}$. 
Moreover, ${}^{\perp}C^{\m} = \mathcal{F}_{\m}^{\intercal}$ in $\rmod R_{\m}$, so $C^{\m}$ is of cofinite type.
\end{lemma}
\begin{proof} The first claim has already been noticed above (see \cite[15.2]{GT}). By \cite[13.50]{GT}, $T_{\m}$ is an $n$-tilting $R_{\m}$-module. Moreover, $C^{\m} = \Hom R{R_{\m}}{T^*} \cong (T_{\m})^* = \Hom {R_{\m}}{T_{\m}}{W_{\m}}$, because $\Hom R{T_\m}{E(R/\p)} = 0$ for all $\p \nsubseteq \m$ by the Nakayama lemma. Since $W_{\m}$ is an injective cogenerator for $\rmod R_{\m}$, applying \cite[15.2]{GT} again, we see that $C^{\m}$ is an $n$-cotilting $R_{\m}$-module.      

Finally, let $M \in \rmod R_{\m}$ and $i \geq 1$. Then $\Ext i{R_{\m}}M{C^{\m}} \cong  \Ext iR{M}C$ by Lemma \ref{cartanei}(a), so
$M \in {}^{\perp}C^{\m}$ iff  $M \in {}^{\perp}C$ iff $M \in \mathcal F ^{\intercal} = {}^{\perp} \mathcal{F}^*$ by the Ext-Tor duality. The latter is equivalent to $M \in {}^{\perp} (\mathcal F ^*)^{\m} = {}^{\perp} (\mathcal F _{\m})^* = \mathcal{F}_{\m}^{\intercal}$ (see Lemma \ref{cartanei}(b)). Since $\mathcal{F}_{\m} \subseteq \rfmod R_{\m}$ consists of $R_{\m}$-modules of projective dimension $\leq n$, the assertion follows.
\end{proof}
 
Unlike localization, the colocalization is not exact in general, because $R_{\m}$ is flat, but it need not be projective. However, colocalization preserves exactness of all short exact sequences of the form $0 \to A \to B \to C \to 0$ with $A$ pure-injective. This is essential for our next lemma: 

\begin{lemma} \label{colocotilt} Let $\m \in \mSpec R$ and $n < \omega$. Let $T$ be an $\mathcal F$-filtered $n$-tilting module, where $\mathcal F = {}^\perp (T^{\perp}) \cap \rfmod R$. Let $C = T^*$. 
Then the $n$-cotilting $R_{\m}$-module $C^{\m}$ is $(\mathcal{F}^*)^{\m}$-cofiltered.
\end{lemma}
\begin{proof} By the assumption, $C$ is  $\mathcal{F}^*$-cofiltered, with a cofiltration $(C_\alpha \mid \alpha \leq \mu)$.
Let $\mu$ be an ordinal and  $\mathcal{I}=\{C_\alpha, g_{\alpha\beta} \mid \alpha \leq \beta \leq \mu\}$ be the corresponding  inverse
system of modules such that $C_0=0$, $C_\alpha=\varprojlim_{ \beta \leq \alpha}C_\beta$ for all limit ordinals $\alpha \leq \mu$, and
$\Ker {g_{\alpha,\alpha+1}}$ is isomorphic to an element of $\mathcal{F}^*$ for all $\alpha \leq \mu$.

Define $$g_{\alpha\beta}^{\m} = \Hom R{R_{\m}}{g_{\alpha\beta}}: (C_\beta)^{\m} \to (C_\alpha)^{\m}$$ for all $\alpha \leq \beta \leq \mu$.
Then $g_{\alpha\beta}^{\m}$ is an epimorphism for all $\alpha \leq \beta \leq \mu$, because $R_{\m}$ is flat and $\Ker {g_{\alpha,\alpha+1}}$
is pure injective. Since the covariant $\Hom R{R_{\m}}-$ functor commutes with inverse limits,
we obtain $C_0^{\m}=0$ and $(C_\alpha)^{\m}=\varprojlim_{\beta \leq \alpha}(C_\beta)^{\m}$ for all limit ordinals $\alpha \leq \mu$.
Also $\Ker {g_{\alpha,\alpha+1}}^{\m}$ is isomorphic to an element of  $(\mathcal{F}^*)^{\m}$ for each $\alpha \leq \mu$.
It follows that $$\mathcal{I}^{\m}=\{(C_\alpha)^{\m}, g_{\alpha\beta}^{\m}:\alpha \leq \beta \leq \mu\}$$ is a continuous inverse system
of  $R_{\m}$-modules witnessing that the colocalized module $C^{\m}$ is $(\mathcal{F}^*)^{\m}$-cofiltered.
\end{proof}

\medskip
In \cite[Theorem 2.5]{SB1}, Bazzoni proved that when $R$ is a commutative domain and $C$ is a $1$-cotilting module, then
the colocalized module $C^{\m}$ is a $1$-cotilting $R_{\m}$-module for each $\m \in \mSpec R$, and $\prod_{\m\in \mSpec R}C^{\m}$ is a $1$-cotilting module equivalent to $C$. The proof easily extends to an arbitrary commutative ring $R$. However, the extension of the result to $n$-cotilting modules over noetherian rings for $n > 1$ needs more work.

\begin{lemma} \label{coloc} Let $n < \omega$ and $C$ be an $n$-cotilting module which is of cofinite type, so ${}^{\perp}C = \mathcal F ^{\intercal}$ for some $\mathcal{F} \subseteq \rfmod R$ consisting of modules of projective dimension $\leq n$. Let $\m \in \mSpec R$. Then as an $R$-module, $C^{\m}$ is pure-injective, and satisfies conditions (C1) and (C2). Moreover, $\Ext iR{C^{\m}}C=0$ for each $i>0$ and each $\m \in \mSpec R$.
\end{lemma}
\begin{proof} The pure-injectivity of $C^{\m}$ as an $R$-module has already been observed in Example \ref{ex}.2. By Lemma \ref{cartanei}(a) each injective $R_{\m}$-module is injective also as $R$-module, so $C^{\m}$ has injective dimension $\leq n$ by condition (C1) for the cotilting $R_{\m}$-module $C^m$. Moreover, for each $i>0$ and all cardinals $\kappa$, we have $\Ext iR{(C^{\m})^\kappa}{C^{\m}} = \Ext i{R_{\m}}{(C^{\m})^\kappa }{C^{\m}} = 0$, so (C2) holds for $C^{\m}$ viewed as an $R$-module. 
The equality $\Ext iR{C^{\m}}C=0$ now follows by Lemma \ref{cartanei}(b).
\end{proof}

\begin{lemma} \label{main} Let $n < \omega$ and $C$ be an $n$-cotilting module which is of cofinite
type. Then the module $D = \prod_{\m\in \mSpec R}C^{\m}$ is $n$-cotilting, and equivalent to $C$.
\end{lemma}
\begin{proof} We will prove that $$\Cog_n(D)={}^{\perp}D={}^{\perp}C=\Cog_n(C).$$
The last equality holds since $C$ is an $n$-cotilting module.

First, we show that $\Cog_n(D)\subseteq \Cog_n(C)={}^{\perp}C$.
Let $M \in \Cog_n(D)$. There is a long exact sequence $$0 \to M \overset{f_0}\hookrightarrow D^{\beta_1} \overset{f_1}\to  \cdots \overset{f_{n-2}}\to D^{\beta_{n-1}} \overset{f_{n-1}}\to D^{\beta_n}.$$
We will show that $\Ext iR{M}C=0$ for each $i\geq 1$.

Consider the short exact sequences (where $i < n-1$ and $K_i =\im {f_i}$)
$$0 \to K_i \to D^{\beta_{i+1}} \to K_{i+1} \to 0.$$
Applying the functor $\Hom R-C$ we obtain the long exact sequences
$$\ldots \to \Ext 1R{K_{i+1}}C \to \Ext 1R{D^{\beta_{i+1}}}C \to \Ext 1R{K_i}C \to \ldots$$
$$\ldots \to \Ext iR{K_{i+1}}C \to \Ext iR{D^{\beta_{i+1}}}C \to \Ext iR{K_i}C \to$$
$$\to \Ext {i+1}R{K_{i+1}}C \to \Ext {i+1}R{D^{\beta_{i+1}}}C \to \Ext {i+1}R{K_{i}}C \to \ldots$$
We know that $D \in \Cog_n(C)$ since $C^{\m} \in \Cog_n(C)$ by Lemma \ref{coloc}. So
$$\Ext 1R{M}C \cong \Ext 2R{K_1}C \cong \dots \cong \Ext {n+1}R{K_n}C = 0$$ because $C$ has injective dimension $\leq n$.
Similarly, $\Ext iR{M}C = 0$ for each $i > 1$.

Next, we show that ${}^{\perp}C \subseteq {}^{\perp}D$. In fact, $M \in {}^{\perp}D$
if and only if $\Ext iRM{C^{\m}}=0$ for each $i\geq1$ and $\m\in \mSpec R$. By Lemma \ref{cartanei}(c), $\Ext iRM{C^{\m}} \cong \Ext iR{M_{\m}}C$.

Let $M \in {}^{\perp}C$. We claim that $M_{\m} \in {}^{\perp}C$. Since each flat module is a  direct limit of
finitely generated free modules, $M_{\m}=M\otimes_R R_{\m}=M\otimes_R \varinjlim_\alpha R^{n_\alpha}= \varinjlim_\alpha M\otimes_R R^{n_\alpha}=\varinjlim_\alpha M^{n_\alpha}$.
Since the cotilting class ${}^{\perp} C$ is closed under direct products and direct limits, the claim follows.

Finally, we show  that $D$ is an $n$-cotilting module. This will give the remaining equality ${}^{\perp}D=\Cog_n(D)$. We proceed by verifying conditions
(C1)-(C3) for $D$.

(C1) Since $C$ is $n$-cotilting, $i.d.(C)\leq n$. By Lemma \ref{coloc},$i.d.(C^{\m}) \leq n$ for each $\m \in \mSpec R$, whence $i.d.(D) \leq n$. 

(C2) We claim that $\Ext iR{D^\kappa}D=0$ for all  $i\geq1$. Since ${}^{\perp}C \subseteq {}^{\perp}D$, it suffices to prove that
$\Ext iR{D^\kappa}C=0$ for all $i > 0$.

It is enough to show that $\Ext iR{D}C=0$ for all $i\geq 1$. By the definition of $D$, this is equivalent to $\Ext iR{\prod_{\m\in \mSpec R}C^{\m}}C=0$. Again,
it suffices to prove that $\Ext iR{C^{\m}}C=0$ for every maximal ideal $\m$, but this holds by Lemma \ref{coloc}.

(C3) Since $C$ is of cofinite type, $C$ is equivalent to $T^*$ for an $n$-tilting module $T$. There exists a long exact sequence of the form
$$0 \to R \to T_0 \to \ldots \to T_n \to 0$$ where $T_i \in \Add(T)$ for all $i \leq n$. Localizing at a maximal ideal $\m$, we obtain the exact sequence
$$0 \to R_{\m} \to (T_0)_{\m} \to \ldots \to (T_n)_{\m} \to 0.$$

Clearly, $(T_i)_{\m} \in \Add(T_{\m})$ for each maximal ideal $\m$, since localization commutes with direct sums.

Applying the exact functor $(-)^* = \Hom R-{\bigoplus_{\m\in \mSpec R} E_{R_{\m}}(R_{\m}/\m_{\m})}$ to the localized exact sequence, we obtain
$$0 \to (T_n)_{\m}^* \to \ldots \to (T_0)_{\m}^* \to R_{\m}^* \to 0$$
where $((T_i)_{\m})^* \in \Prod(T_{\m}^*)$.

Since  $(C^\prime)^{\m} \in \Prod C^{\m}$ where $C^\prime =T^*$, also $((T_i)_{\m})^* \in \Prod (C^\prime)^{\m}$ for all $i\leq n$.

From the previous exact sequence we obtain
$$0 \to \prod_{\m\in \mSpec R}((T_n)_{\m})^* \to \ldots \to \prod_{\m\in \mSpec R}((T_0)_{\m})^* \to \prod_{\m\in \mSpec R} (R_{\m})^* \to 0.$$
So for each $i \leq n$, $\prod_{\m\in \mSpec R}((T_i)_{\m})^* \leq_{\oplus}\prod_{\m\in \mSpec R}(C^{\m})^{\alpha_i}=D^{\alpha_i}$ for a cardinal $\alpha_i$.

It remains to show that $E= \prod_{\m\in \mSpec R}(R_{\m})^* $ is an injective cogenerator for $\rmod R$. However, $R_{\m}$ is flat, hence $(R_{\m})^*$ is injective, and so is $E$.
Moreover, $(R_{\m})^*=\Hom R{R_{\m}}{\bigoplus_{\m\in \mSpec R} E_{R_{\m}}(R_{\m}/\m_{\m})}$ has a direct summand $\Hom R{R_{\m}}{E_{R_{\m}}(R_{\m}/\m_{\m})} = \Hom {R_{\m}}{R_{\m}}{E_{R_{\m}}(R_{\m}/\m_{\m})} = E_{R_{\m}}(R_{\m}/\m_{\m})$, and
the latter module contains the simple module $R/\m$. It follows that $E$ is a cogenerator for $\rmod R$.
\end{proof}

Lemmas \ref{coloc} and \ref{main} yield

\begin{theorem} \label{yield} Let $n < \omega$ and $C$ be an $n$-cotilting module of cofinite type. Then for each $\m \in \mSpec R$, $C^{\m}$ is an $n$-cotilting $R_{\m}$-module of cofinite type, and $D = \prod_{\m\in \mSpec R}C^{\m}$ is an $n$-cotilting module equivalent to $C$.
\end{theorem}

\section{The noetherian case}

\emph{In this section, we will assume that $R$ is a commutative noetherian ring.} Then all cotilting modules are of cofinite type by \cite[Theorem 4.2]{APST}. We are going to investigate the relation among the $R_{\m}$-modules $C^{\m}$ ($\m \in \mSpec R$) from  Theorem \ref{yield}, the goal being to obtain a complete characterization of cotilting modules in terms of compatible families of their colocalizations.

First, we recall that in the noetherian setting, cotilting classes correspond 1-1 to characteristic sequences of sets of prime ideals:

\begin{definition} \label{characts} Let $n < \omega$ and $\mathcal P = (P_0,\dots,P_{n-1})$ be sequence of subsets of $\Spec R$ such that
\begin{itemize}
\item[\rm{(i)}] $P_i$ is a \emph{lower subset} of $\Spec R$ for each $i < n$ (i.e., $\q \subseteq \p \in P_i$ implies $\q \in P_i$ for all $\q \in \Spec R$);
\item[\rm{(ii)}] $P_i \subseteq P_{i+1}$ for all $i < n-1$;
\item[\rm{(iii)}] $\Ass \Omega ^{-i}(R) \subseteq P_i$ for each $i < n$.
\end{itemize}
Then $\mathcal P$ is called a \emph{characteristic sequence} (of length $n$) in $\Spec R$.
\end{definition}

Note that if $P$ is a lower subset of $\Spec R$ and  $( M_\alpha \mid \alpha < \kappa )$ is a family of modules such that $\Ass M_\alpha \subseteq P$ for all $\alpha$, then $\Ass \prod_{\alpha < \kappa} M_\alpha \subseteq P$. Indeed, for each $\q \in \Ass \prod_{\alpha < \kappa} M_\alpha$, there is a non-zero homomorphism from $E(R/{\q})$ to $E(R/{\p})$ for some $\p \in P$, whence $\q \subseteq \p$ and $\q \in P$. Similarly, for each $i < \omega$, if  $\Ass \Omega ^{-i} M_\alpha \subseteq P$ for all $\alpha < \kappa$, then $\Ass \Omega ^{-i} (\prod_{\alpha < \kappa} M_\alpha) \subseteq P$.

Also notice that if $R$ is Gorenstein, then condition (iii) just says that for each $i \leq i.d.(R)$, the set $P_i$ contains all prime ideals of height $i$.

\medskip
The following result was proved in \cite{APST} (see also \cite[\S 16]{GT} and \cite{DT}). Here, for $i < n < \omega$ and an $n$-cotilting class $\mathcal C$, we define $\mathcal C _{(i)} = \{ M \in \rmod R \mid  \Omega^j (M) \in \mathcal C \hbox{ for all } j > i \}$ (so in particular, $\mathcal C _{(0)} = \mathcal C$). Notice that if $C$ is an $n$-cotilting module such that $\mathcal C = {}^{\perp} C$, then $\mathcal C _{(i)} = \{ M \in \rmod R \mid  \Ext jRMC = 0 \hbox{ for all } j > i \}$; this is an $(n-i)$-cotilting class (see \cite[15.13]{GT}).

\begin{lemma} \label{apst} \cite{APST} Let $n < \omega$. There is a 1-1 correspondence between $n$-cotilting classes $\mathcal C$ and characteristic sequences $\mathcal P = (P_0,\dots ,P_{n-1})$ of length $n$ in $\Spec R$ given by the mutually inverse assignments
$$\mathcal C \mapsto \mathcal P _{\mathcal C} = (\Ass \mathcal C _{(0)}, \dots , \Ass \mathcal C _{(n-1)})$$ and
$$\mathcal P \mapsto \mathcal C _{\mathcal P} = \{ M \in \rmod R \mid  \Ass \Omega ^{-i} M \subseteq P_i \hbox{ for all } i < n \}.$$
Each $n$-cotilting class $\mathcal C$ is of cofinite type. The $n$-tilting class $\mathcal T$ corresponding to the characteristic sequence $\mathcal P = (P_0,\dots ,P_{n-1})$ is
$$\mathcal T = \{ N \in \rmod R \mid  \Tor RiM{R/\p} = 0 \mbox{ for all } p \in \Spec R \setminus P_i \hbox{ and } i < n \}.$$
Moreover, for each $i < n$ and $\p \in \Spec R$, we have $\p \in \Ass \mathcal C _{(i)}$, iff $E(R/{\p}) \in \mathcal C _{(i)}$.
\end{lemma}

If $C$ is a cotilting module, then $C^{\m}$ is a cotilting $R_{\m}$-module by Lemma \ref{coloc}. We will now consider relations between the corresponding characteristic sequences of subsets in $\Spec R$ and $\Spec {R_{\m}}$. First, we need more notation:

\begin{definition} \label{compat}
\begin{itemize}
\item[\rm{(i)}]  Let $\p \in \Spec {R_{\m}}$. Then $\widehat{\p}$ denotes the prime ideal of $R$ such that $\widehat{\p} \subseteq \m$ and $(\widehat{\p})_{\m} = \p$. Similarly, for a subset $P \subseteq \Spec {R_{\m}}$, $\widehat{P} = \{ \widehat{\p} \mid \p \in P \}$.
\item[\rm{(ii)}] Let $n < \omega$. Then $\mathfrak P$ is a \emph{compatible family of characteristic sequences} of length $n$ provided that $\mathfrak P = ( \mathcal P_{\m} \mid \m \in \mSpec R )$ where for each $\m \in \mSpec R$, $\mathcal P _{\m} = (P_{0,\m}, \dots ,P_{n-1,\m})$ is a characteristic sequence of length $n$ in $\Spec {R_{\m}}$, and $\widehat{P_{i,\m}}$ and $\widehat{P_{i,\m^\prime}}$ contain the same prime ideals from the set $\{ \p \in \Spec R \mid \p \subseteq \m \cap \m ^\prime \}$, for all $\m, \m^\prime \in \mSpec R$ and $i < n$.
\item[\rm{(iii)}]  Let $n < \omega$. Let $\mathfrak C = ( C(\m) \mid \m \in \mSpec R )$ be a family such that $C(\m)$ is an $n$-cotilting $R_{\m}$-module for each $\m \in \mSpec R$. Let $\mathcal P _{\m} = (P_{0,\m}, \dots ,P_{n-1,\m})$ be the characteristic sequence corresponding to the $n$-cotilting class $\mathcal C _{\m} = {}^{\perp} C(\m)$ in $\rmod R_{\m}$. Then $\mathfrak C$ is a \emph{compatible family of $n$-cotilting modules} provided that $( \mathcal P_{\m} \mid \m \in \mSpec R )$ is a compatible family of characteristic sequences of length $n$.

Two compatible families of $n$-cotilting modules, $\mathfrak C = ( C(\m) \mid \m \in \mSpec R )$ and $\mathfrak C ^\prime = ( C^\prime(\m) \mid \m \in \mSpec R )$, are called \emph{equivalent} provided that the $n$-cotilting $R_{\m}$-modules $C(\m)$ and $C^\prime(\m)$ are equivalent for each $\m \in \mSpec R$. 
\end{itemize}
\end{definition}

Notice that two compatible families $\mathfrak C$ and $\mathfrak C ^\prime$ of $n$-cotilting modules are equivalent, if and only if the corresponding compatible families of characteristice sequences $\mathfrak P$ and $\mathfrak P ^\prime$ are equal. 

\begin{lemma} \label{charloc} Let $n < \omega$, $\m \in \mSpec R$, and $C$ be an $n$-cotilting module. Let $\mathcal P = ( P_0, \dots, P_{n-1} )$ be the corresponding characteristic sequence in $\Spec R$. 
For each $i < n$, let $P_{i,\m} = \{ \p_{\m} \mid \p \in P_i {\,\&\,} \p \subseteq \m \}$. Then $\mathcal P _{\m} = ( P_{0,\m}, \dots , P_{n-1,\m})$ is a characteristic sequence in $\Spec {R_{\m}}$ which corresponds to the $n$-cotilting $R_{\m}$-module $C^{\m}$. Moreover, the families $\mathfrak P = ( \mathcal P_{\m} \mid \m \in \mSpec R )$ and $\mathfrak C = ( C^{\m} \mid \m \in \mSpec R )$ are compatible.
\end{lemma}
\begin{proof} It is easy to see that $\mathcal P _{\m}$ satisfies conditions (i) and (ii) from Definition \ref{characts} for $R_{\m}$. Consider $i < n$ and let $\p \in \Ass \Omega ^{-i}(R_{\m})$. Since the minimal injective coresolution of $R$ localizes to the minimal injective coresolution of $R_{\m}$ (see e.g.\ \cite[\S 18]{M}), $\widehat{\p} \in \Ass \Omega ^{-i}(R) \subseteq P_i$ whence $\p = (\widehat{\p})_{\m} \in P_{i,\m}$. Thus condition (iii) holds for $\mathcal P _{\m}$.

Let $i < n$ and $\p \in \Spec R$ be such that $\p \subseteq \m$. Denote by $\mathcal C _{\m}$ the $n$-cotilting class in $\rmod R_{\m}$ induced by $C^{\m}$. By Lemma \ref{apst}, ${\p}_{\m} \in \Ass (\mathcal C _{\m})_{(i)}$ iff $E_{R_{\m}}(R_{\m}/\p_{\m}) \in (\mathcal C _{\m})_{(i)}$ iff  $\Omega^j(E_{R_{\m}}(R_{\m}/\p_{\m})) \in \mathcal C _{\m}$ for all $j > i$. The latter is equivalent to $\Ext j{R_{\m}}{E_{R_{\m}}(R_{\m}/\p_{\m})}{C^{\m}} = 0$ for all $j > i$, and to $\Ext jR{E(R/\p)}C= 0$ for all $j > i$ by Lemma \ref{cartanei}(b).

Similarly, $\p \in P_i = \Ass \mathcal C _{(i)}$ iff $\Ext jR{E(R/\p)}C= 0$, for all $\p \in \Spec R$ and $j > i$. Thus $\widehat{P_{i,\m}} = \Ass (\mathcal C _{\m})_{(i)}$ for each $i < n$. It follows that $\mathcal P _{\m}$ is a characteristic sequence corresponding to $\mathcal C _{\m}$, and that the family $\mathfrak P$ is compatible. Then $\mathfrak C$ is compatible by Definition \ref{compat}(iii).
\end{proof}

Now, we can extend Theorem \ref{yield} to a 1-1 correspondence between cotilting $R$-modules and compatible families of cotilting $R_{\m}$-modules ($\m \in \mSpec R$).

\begin{theorem} \label{fullch} Let $n < \omega$ and $C \in \rmod R$.
\begin{itemize}
\item[\rm{(i)}] If $C$ is an $n$-cotilting module, then $\mathfrak C _C = (C^{\m} \mid  \m \in \mSpec R )$ is a compatible family of $n$-cotilting $R_{\m}$-modules.
\item[\rm{(ii)}] Let $\mathfrak C = (C (\m) \mid  \m \in \mSpec R )$ be a compatible family of $n$-cotilting $R_{\m}$-modules. Then the module $D_{\mathfrak C} = \prod_{\m\in \mSpec R} C(\m)$ is $n$-cotilting. The characteristic sequence corresponding to $D_{\mathfrak C}$ is $\mathcal P = (P_0, \dots , P_{n-1})$ where $P_i = \bigcup_{\m \in mSpec R} \widehat{P_{i,\m}}$ for each $i < n$, and $(P_{0,\m}, \dots ,P_{n-1, \m})$ is the characteristic sequence corresponding to $C(\m)$.
\item[\rm{(iii)}] If $C$ is $n$-cotilting, then $D_{\mathfrak C _C}$ is an $n$-cotilting module equivalent to $C$.
\item[\rm{(iv)}] If  $\mathfrak C = (C(\m) \mid  \m \in \mSpec R )$ is a compatible family of $n$-cotilting $R_{\m}$-modules and $\mathfrak C _{D _{\mathfrak C}} = ( (D_{\mathfrak C})^{\m} \mid  \m \in \mSpec R )$, then the families $\mathfrak C$ and $\mathfrak C _{D _{\mathfrak C}}$ are \emph{equivalent} (that is, $(D_{\mathfrak C})^{\m}$ is an $n$-cotilting $R_{\m}$-module equivalent to $C(\m)$ for each $\m \in \mSpec R$).
\end{itemize}
\end{theorem}
\begin{proof} (i) This follows by Lemma \ref{charloc}.

(ii) The compatibility of $\mathfrak C$ implies that $\mathcal P$ is a characteristic sequence in $\Spec R$: conditions (i) and (ii) of Definition \ref{characts} are clearly satisfied, and condition (iii) follows from the fact that the minimal injective coresolution of $R$ localizes to the minimal injective coresolution of $R_{\m}$ for each $\m \in \mSpec R$.

We will prove that $D_{\mathfrak C}$ is an $n$-cotilting module corresponding to $\mathcal P$.

Since each injective $R_{\m}$-module is injective also as $R$-module, $D_{\mathfrak C}$ is a product of modules of injective dimension $\leq n$, whence condition (C1) holds for $D_{\mathfrak C}$.

Further, for each $\m \in \mSpec R$, there is a long exact sequence
$$\mathcal E _{\m} : 0 \to C_{n,\m} \to \dots \to C_{0,\m} \to E_{R_{\m}}(R_{\m}/\m_{\m}) \to 0$$
where $C_{i,\m} \in \Prod C(\m)$ as an $R_{\m}$-module. Since $C(\m) \in \Prod D_{\mathfrak C}$ as an $R$-module, and $W = \prod_{\m \in \mSpec R} E(R/\m)$ is an injective cogenerator for $\rmod R$, the sequence $\mathcal E = \prod_{\m \in \mSpec R} \mathcal E _{\m}$ witnesses condition (C3) for $D_{\mathfrak C}$.

Let $\mathcal C _{\mathcal P}$ denote the $n$-cotilting class corresponding to $\mathcal P$ by Lemma \ref{apst}. Assume that $M \in \mathcal C _{\mathcal P}$. We claim that $\Ext iR{M}{D_{\mathfrak C}} = 0$ for each $i \geq 1$, that is, $\Ext iRM{C(\m)}= 0$ for all $i \geq 1$ and $\m \in \mSpec R$.

By assumption $\Ass \Omega ^{-i} M \subseteq P_i$, hence $\Ass \Omega ^{-i} M_{\m} \subseteq P_{i,\m}$ for each $\m \in \mSpec R$, by the compatibility of $\mathfrak C$. So $M_{\m}$ belongs to the $n$-cotilting class induced by $C(\m)$. By Lemma \ref{cartanei}(a), we infer that $\Ext iRM{C(\m)} \cong \Ext i{R_{\m}}{M_{\m}}{C(\m)}=0$ for all $i \geq 1$ and $\m \in \mSpec R$, and the claim is proved. In other words, $\mathcal C _{\mathcal P} \subseteq {}^{\perp} D_{\mathfrak C}$.

In order to prove condition (C2) for $D_{\mathfrak C}$, it only suffices to show that $(D_{\mathfrak C})^\kappa \in \mathcal C _{\mathcal P}$ for all cardinals $\kappa$. Since all the sets $P_i$ are lower subsets of $\Spec R$, it suffices to prove that $\Ass \Omega ^{-i} C(\m) \subseteq P_i$ for all  $i < n$ and all $\m \in \mSpec R$. Consider a minimal injective coresolution $\mathcal I$ of $C(\m)$ in $\rmod R _{\m}$. By assumption, $\Ass{_{R_{\m}}\Omega ^{-i} C(\m)} \subseteq P_{i,\m}$. Since $\mathcal I$ is also an injective coresolution in $\rmod R$, we have $\Ass \Omega ^{-i} C(\m) \subseteq \widehat{P_{i,\m}} \subseteq P_i$. This proves condition (C2).

Finally, assume that $M$ is a module such that $\Ext iRM{D_{\mathfrak C}} = 0$ for all $i \geq 1$. Let $\m \in \mSpec R$. Then $\Ext iRM{C(\m)}= 0$ for all $i \geq 1$. Lemma \ref{cartanei}(a) gives $\Ext i{R_{\m}}{M_{\m}}{C(\m)}= 0$, and hence $\Ass{_{R_{\m}}\Omega ^{-i} M_{\m}} \subseteq P_{i,\m}$, for all $i \geq 1$.

Consider a minimal injective coresolution $\mathcal I$ of $M$ in $\rmod R$. Localizing at $\m$, we obtain a minimal injective coresolution of $M_{\m}$ in $\rmod R_{\m}$.
Since $\Ass{_{R_{\m}}\Omega ^{-i} M_{\m}} \subseteq P_{i,\m}$, we infer that $\Ass \Omega ^{-i} M \subseteq \bigcup_{\m \in mSpec R} \widehat{P _{i,\m}} = P_i$.
Thus ${}^{\perp} D_{\mathfrak C} \subseteq \mathcal C _{\mathcal P}$, and $\mathcal P$ is the characteristic sequence of $D_{\mathfrak C}$.

(iii) This follows by Lemma \ref{main}.

(iv) By part (ii), the characteristic sequence of $D_{\mathfrak C}$ is $\mathcal P = (P_0, \dots ,P_{n-1})$, where  $P_i = \bigcup_{\m \in mSpec R} 
\widehat{P_{i,\m}}$ for each $i < n$, and $(P_{0,\m}, \dots ,P_{n-1, \m})$ denotes the characteristic sequence corresponding to $C(\m)$. By Lemma \ref{charloc}, $(D_{\mathfrak C})^{\m}$ is an $n$-cotilting module with the characteristic sequence $Q_{i,\m} = \{ \q_{\m} \mid \q \in P_i {\,\&\,} \q \subseteq \m \}$. By compatibility, $Q_{i,\m} = P_{i,\m}$ for all $i < n$, so $C(\m)$ is equivalent to $(D_{\mathfrak C})^{\m}$ for each $\m \in \mSpec R$.
\end{proof}

\begin{corollary} \label{coloccor} Let $n < \omega$. There is a 1-1 correspondence between equivalence classes of $n$-cotilting $R$-modules and equivalence classes of compatible families of $n$-cotilting $R_{\m}$-modules.

The equivalence class of an $n$-cotilting module $C$ corresponds to the equivalence class of the family of its colocalizations $( C^{\m} \mid \m \in \mSpec R )$. Conversely, the equivalence class of a compatible family $( C(\m) \mid \m \in \mSpec R )$ corresponds to the equivalence class of the $n$-cotilting module $\prod_{\m \in \mSpec R} C(\m)$.
\end{corollary}

We finish by translating Corollary \ref{coloccor} to the tilting setting. First, we have to introduce the corresponding notion of a compatible family:

\begin{definition} \label{compatilt}
Let $n < \omega$. Let $\mathfrak T = ( T(\m) \mid \m \in \mSpec R )$ be a family such that $T(\m)$ is an $n$-cotilting $R_{\m}$-module for each $\m \in \mSpec R$. Let $\mathcal P _{\m} = (P_{0,\m}, \dots ,P_{n-1,\m})$ be the characteristic sequence corresponding to the $n$-cotilting class $\mathcal T _{\m} = T(\m)^{\perp}$ in $\rmod R_{\m}$ by Lemma \ref{apst}. Then $\mathfrak T$ is a \emph{compatible family of $n$-tilting modules} provided that 
$\mathfrak P = ( \mathcal P_{\m} \mid \m \in \mSpec R )$ is a compatible family of characteristic sequences of length $n$.

Two compatible families of $n$-tilting modules, $\mathfrak T = ( T(\m) \mid \m \in \mSpec R )$ and $\mathfrak T ^\prime = ( T^\prime(\m) \mid \m \in \mSpec R )$, are called \emph{equivalent} provided that the $n$-tilting $R_{\m}$-modules $T(\m)$ and $T^\prime(\m)$ are equivalent for each $\m \in \mSpec R$ (that is, the corresponding compatible families of characteristic sequences $\mathfrak P$ and $\mathfrak P ^\prime$ coincide).  
\end{definition}   

\begin{corollary} \label{loccor} Let $n < \omega$. There is a 1-1 correspondence between equivalence classes of $n$-tilting $R$-modules and equivalence classes of compatible families of $n$-tilting $R_{\m}$-modules.
The equivalence class of an $n$-tilting module $T$ corresponds to the equivalence class of the family of its localizations $( T_{\m} \mid \m \in \mSpec R )$. 
\end{corollary}  
\begin{proof} By Lemma \ref{apst}, (equivalence classes of) $n$-tilting modules correspond 1-1, via the duality $(-)^*$, to (equivalence classes of) cotilting modules. Moreover, by \cite[15.18]{GT}, two $n$-tilting modules $T$ and $T^\prime$ are equivalent, iff the dual modules $C = T^*$ and $C^\prime = (T^\prime)^*$ are equivalent as $n$-cotilting modules. By Corollary \ref{coloccor}, there is a further 1-1 correspondence to (equivalence classes of) compatible families of colocalizations $( C^{\m} \mid \m \in \mSpec R )$. Since $C^{\m}$ is isomorphic to $(T_{\m})^*$ by Lemma \ref{dual} for each $\m \in \mSpec R$, applying Lemma \ref{apst} and \cite[15.18]{GT} again, we finally proceed to (equivalence classes of) compatible families of localizations of $T$. 
\end{proof}

\begin{remark}\label{remfin} \rm
The 1-1 correspondence in Corollary \ref{coloccor} is easy to compute in both directions. In particular, the cotilting module $C$ is recovered up to equivalence as the product of its colocalizations: $\prod_{\m \in \mSpec R} C^{\m}$. There does not appear to be any simple way of computing the tilting module $T$ from the family of its localizations $( T_{\m} \mid \m \in \mSpec R )$: the naive idea of taking $T^\prime = \bigoplus_{\m \in \mSpec R} T_{\m}$ fails already for $n = 0$ (where $T$ is a projective generator, but $T^\prime$ is flat and non-projective, hence not tilting; in fact, though $(T^\prime)^*$ is a cotilting module equivalent to $T^*$, $T^\prime \notin \Add T$).  
\end{remark}

\end{document}